\documentclass[12pt,a4paper]{amsart}
\usepackage{graphicx}
\usepackage{float,epsfig}
\usepackage[english]{babel}
\usepackage[psamsfonts]{amssymb}
\usepackage{amsfonts}
\usepackage{amsmath}
\usepackage[utf8]{inputenc}
\newtheorem{theorem}{Theorem}

\newtheorem{corollary}{Corollary}

\def\Re{\mathop{\rm Re}\nolimits}
\def\Im{\mathop{\rm Im}\nolimits}
\def\th{\mathop{\rm th}\nolimits}

\begin{document}

\title{DISTORTION OF THE TRIANGULAR RATIO METRIC UNDER MOEBIUS TRANSFORMATIONS}

\author[S.~Nasyrov]{S.~Nasyrov}
\address{Kazan Federal University,
         Kazan, Russia}
 \email[]{semen.nasyrov@yandex.ru}

\maketitle

\begin{abstract}
		Let $\mathbb{U}$ be the unit disk in the complex plane. Denote by $s_\mathbb{U}(x,y)$ the triangular ratio metric in $\mathbb{U}$; for $x$, $y\in \mathbb{U}$ the value of $s_\mathbb{U}(x,y)$ equals the ratio of the Euclidean distance $|x-y|$ to the value $\inf_{z\in \partial \mathbb{U}}(|x-z|+|z-y|)$.   In the monograph by P.~Hariri, R.~Kl\'en,  and M.~Vuorinen "Conformally invariant metrics and quasiconformal mappings" (2020) the following problem was stated: for every Moebius automorphism of the unit disk, $w=f(z)=\frac{z+a}{1+za}$, $0\le a<1$,  and every points $z_1$, $z_2\in \mathbb{U}$ the sharp inequality $s_\mathbb{U}(f(z_1),f(z_2))\le (1+a)s_\mathbb{U}(z_1,z_2)$ holds. We prove that the conjecture is valid.
\end{abstract}

\noindent {Keywords:} {hyperbolic metric, triangular ratio metric, Moebius transformation}.

\noindent {Mathematics Subject Classification:} 51M09; 51M16; 30C20.

 %51M09 Elementary problems in hyperbolic and elliptic geometries
 %51M16 Inequalities and extremum problems
 %30C20 Conformal mappings of special domains
 %30C75 Extremal problems for conformal and quasiconformal mappings, other methods
 %30A10 Inequalities in the complex domain

\maketitle

\section*{Introduction}

In the theory of planar conformal and quasiconformal mappings the hyperbolic metric plays an important role (see, e.g. \cite{goluzin,avv,HKV}).   It is invariant under conformal mappings and decreases  as  the domain expands. Disadvantage  of the hyperbolic metric is the fact that it is difficult to calculate in domains with complicated geometry. In Euclidean spaces the class of conformal mappings is narrow and the study of quasiconformal mappings needs the use of other metrics whose properties are close to those of the hyperbolic metric in the plane. Such metrics depends not only on the Euclidean distance between points but also on their location with respect to the boundary of the domain under consideration; they are called intrinsic metrics.

One of such metrics is the triangular ratio metric. For a given domain $D\subset \mathbb{R}^n$ with a non-degenerate boundary $\partial D$ the triangular ratio metric $s_D$ can be defined as follows.
If $u$, $v\in D$, then $$s_D(u,v)=\frac{|u-v|}{\inf_{w\in \partial D}(|u-w|+|w-v|)}.$$ Here $|u-v|$ denotes the Euclidean distance between $u$ and $v$.

The properties of the triangular ratio metric for some classes of domains and its applications in the theory of conformal and quasiconformal mappings were studied in works of J.~Chen, P.~Hariri,  R.~Kl\'en, P.~H\"ast\"o, O.~Rainio, M.~Vuorinen and others (see, e.g., \cite{dknv,ranio,ranio2,ranio_vuor2,ranio_vuor}).

In the monograph \cite[p.455, Problem~(18)]{HKV} (see also \cite[thrm.~1.5]{chkv}) the problem on distortion of the triangular ratio metric in the unit disk  $\mathbb{U}=\{|z|<1\}$ under Moebius transformation
\begin{equation}\label{moeb}
w=f(z)=\frac{z+a}{1+za}, \quad 0\le a<1,
\end{equation}
was stated.
It was shown  that
\begin{equation}\label{1}
s_\mathbb{U}(w_1,w_2)\le L(a)s_\mathbb{U}(z_1,z_2), \quad w_1=f(z_1),\ w_2=f(z_2),
\end{equation}
for all points $z_1$, $z_2\in \mathbb{U}$ with a constant $L(a)\ge 1+a$ and was conjectured that the best value of $L(a)$ is equal to $1+a$.
We will prove that this conjecture holds.

\begin{theorem}\label{th1}
Let $w=f(z)$ be a Moebius automorphism of the unit disk $\mathbb{U}$  of the form \eqref{moeb}. Then for the triangular ratio metric  $s_\mathbb{U}$ in $\mathbb{U}$ and every points $z_1$, $z_2\in \mathbb{U}$ we have
\begin{equation}\label{main}
s_\mathbb{U}(w_1,w_2)\le (1+a)s_\mathbb{U}(z_1,z_2), \quad w_1=f(z_1), w_2=f(z_2),
\end{equation}
and the constant $1+a$ is the best possible.
\end{theorem}

\section{Hyperbolic metric and triangular ratio metric}

First we recall that in the unit disk $\mathbb{U}$ the hyperbolic metric $\rho_\mathbb{U}$ can be defined by the equality
$$
\th\frac{\rho_\mathbb{U}(u,v)}{2}=\frac{|u-v|}{|1-u\overline{v}|}, \quad u,v\in \mathbb{U}.
$$
If $D$ is a simply-connected domain in $\mathbb{C}$ with nondegenerate boundary and $f:D\to U$ is a conformal mapping of $D$ onto $\mathbb{U}$, then, by definition, the hyperbolic metric $\rho_D$ in $D$ is
$$
\rho_D(u,v)=\rho_\mathbb{U}(f(u),f(v)), \quad u,v\in D.
$$
This definition is correct, i.e. it does not depend on the choice of the conformal mapping $f$. As a corollary, we obtain that the hyperbolic metric is invariant under conformal mappings: if $g:D_1\to D_2$ is a conformal mapping of $D_1$ onto $D_2$, then
$$
\rho_{D_1}(u,v)=\rho_{D_2}(g(u),g(v)), \quad  u,v\in D_1.
$$

In particular, in the upper half-plane $\mathbb{H}=\{z\in \mathbb{C}\mid \Im z>0\}$ we have
$$
\th\frac{\rho_\mathbb{H}(u,v)}{2}=\frac{|u-v|}{|u-\overline{v}|}, \quad u,v\in \mathbb{H},
$$
and in the disk $K(z_0,R)=\{z\in \mathbb{C} : |z-z_0|<R\}$
\begin{equation}\label{K}
\th\frac{\rho_{K(z_0,R)}(u,v)}{2}=\frac{R|u-v|}{|R^2-(u-z_0)(\overline{v}-\overline{z}_0)|}, \quad u,v\in K(z_0,R).
\end{equation}
The hyperbolic metric strictly decreases as domain $D$ expands, i.e.  Hyperbolic Metric Principle holds   (see e.g. \cite[ch.~VIII, \S~2]{goluzin}).

\begin{theorem} \label{princ}
If $D_1$ and $D_2$ are two domains in $\mathbb{C}$ with nondegenerate boundaries and $D_1\subset D_2$, then $\rho_{D_1}(u,v)\ge \rho_{D_2}(u,v)$ for every pair of points $u$, $v\in D_1$. Moreover, if  $D_1\not=D_2$ and $u\neq v$, then $\rho_{D_1}(u,v)> \rho_{D_2}(u,v)$.
\end{theorem}

In convex domains the triangular ratio metric can be calculated via hyperbolic metrics of supporting half-planes.

\begin{theorem} \label{supp}
For every convex domain $D\subset \mathbb{C}$ with nonempty boundary
we have
\begin{equation}\label{convex}
s_D(u,v)=\sup_{H} \th\frac{\rho_{H}(u,v)}{2}
\end{equation}
where supremum is taken over all half-planes $H$ containing $D$ and bounded by supporting lines of~$D$.

Moreover, if $G$ is the maximal ellipse with foci $u$, $v$ contained in $D$ and $w\in \partial G\cap \partial D$ , then \\
(I) The triangular ratio distance between $u$ and $v$ is equal to
$$
s_D(u,v)=\th\frac{\rho_{H_w}(u,v)}{2}
$$
where $H_w$ is the  half-plane containing $D$ and bounded by the unique supporting line of~$D$ at the point~$w$,\\
(II) The normal to $G$ at the point $w$ is the bisector of the angle between the vectors $\overrightarrow{wu}$ and $\overrightarrow{wv}$.
\end{theorem}

\begin{proof} First we note that (II) is a simple corollary of the optical property of ellipses or so-called reflection law which states that the focal radii of a point on an ellipse form equal angles with the tangent to the ellipse at that point (see, e.g. \cite{zw}).

Since for every  half-plane $H$ containing $D$ and bounded by a supporting line of~$D$ by Theorem~\ref{princ} we have  $s_D(u,v)\ge\th\frac{\rho_{H}(u,v)}{2}$, to prove  \eqref{convex} we only need to show that (I) is valid. The ellipse $G$ is contained in $H_w$, therefore, it touches $\partial H_w$ at the point $w$. Without loss of generality we can consider that $H_w$ is the upper half-plane $\mathbb{H}$. Then, by (II), $\arg(u-w)=\pi-\arg(v-w)$ and $|u-w|+|w-v|=|u-w|+|w-\overline{v}|=|u-\overline{v}|$. This implies
$$
s_D(u,v)=\frac{|u-v|}{|u-w|+|w-v|}=\frac{|u-v|}{|u-\overline{v}|}=\th\frac{\rho_{\mathbb{H}}(u,v)}{2}.
$$
\end{proof}

As a corollary of Theorem~\ref{supp} we obtain

\begin{theorem} \label{supp1}
In the unit disk $\mathbb{U}$ we have
$$
s_\mathbb{U}(z_1,z_2)=\sup_{0\le\vartheta\le 2\pi} \th\frac{\rho_{\mathbb{H}_\vartheta}(z_1,z_2)}{2},
$$
where $\rho_{\mathbb{H}_\vartheta}$ is the hyperbolic metric in the half-plane ${\mathbb{H}_\vartheta}=\{z\in \mathbb{C} \mid \Re(1-e^{-i\vartheta}z)>0\}$ containing $\mathbb{U}$ and bounded by the supporting line of $\mathbb{U}$ at  the point $e^{i\vartheta}$.
\end{theorem}

\section{Proof of the main theorem}

Let $w_1$, $w_2\in \mathbb{U}$ and $G$ be the maximal ellipse with foci at the points $w_1$ and $w_2$, the interior  of which is in $\mathbb{U}$. Denote by $e^{i\phi}$  a point at which the ellipse touches the unit circle. Let $e^{i\theta}=f^{-1}(e^{i\phi})$.  Then, by Theorem~\ref{supp} and~\ref{supp1}
\begin{equation}\label{sw}
s_{\mathbb{U}}(w_1,w_2)=\th\frac{\rho_{\mathbb{H}_\phi}(w_1,w_2)}{2}
\end{equation}
and
\begin{equation}\label{sz}
s_\mathbb{U}(z_1,z_2)\ge \th\frac{\rho_{\mathbb{H}_\theta}(z_1,z_2)}{2}.
\end{equation}

Consider the preimage $D_\phi=f^{-1}(\mathbb{H}_\phi)$ of $\mathbb{H}_\phi$ under the mapping $f$. It is bounded by the curve $\Gamma_{\theta}$ which is the preimage of the supporting line of the unit disk at the point $e^{i\phi}$. We note that  $\Gamma_{\theta}$ is either a circle or a straight line. Moreover it passes through the point $f^{-1}(\infty)=-1/a$ and touches the unit circle at the point
$e^{i\theta}$. If  $\Gamma_{\theta}$ is a circle, the cases of internal and external tangency of $\Gamma_{\theta}$ and $\partial \mathbb{U}$ are possible (see Fig.~\ref{support}).
Consider two cases.

(a) If  $\Gamma_{\theta}$ is a straight line or a circle with external tangency of $\Gamma_{\theta}$ we have $\mathbb{H}_\theta\subset D_\phi$, therefore, because of conformal invariance of hyperbolic metric, by Theorem~\ref{princ}
we have
$$
\rho_{\mathbb{H}_\phi}(w_1,w_2)=\rho_{{D}_\phi}(z_1,z_2)\le \rho_{\mathbb{H}_\theta}(z_1,z_2),
$$
and this implies that
$$
\th\frac{\rho_{\mathbb{H}_\phi}(w_1,w_2)}{2}\le \th\frac{\rho_{\mathbb{H}_\theta}(z_1,z_2)}{2}.
$$
Taking into account \eqref{sw} and \eqref{sz}, we obtain
\begin{equation}\label{2a}
s_\mathbb{U}(w_1,w_2)\le s_\mathbb{U}(z_1,z_2),
\end{equation}
i.e. in this case we proved a stronger inequality than \eqref{main}  with the constant $1$ in the right-hand side,  instead of $1+a$.

\begin{figure}[ht] \centering
\includegraphics[width=3.5 in]{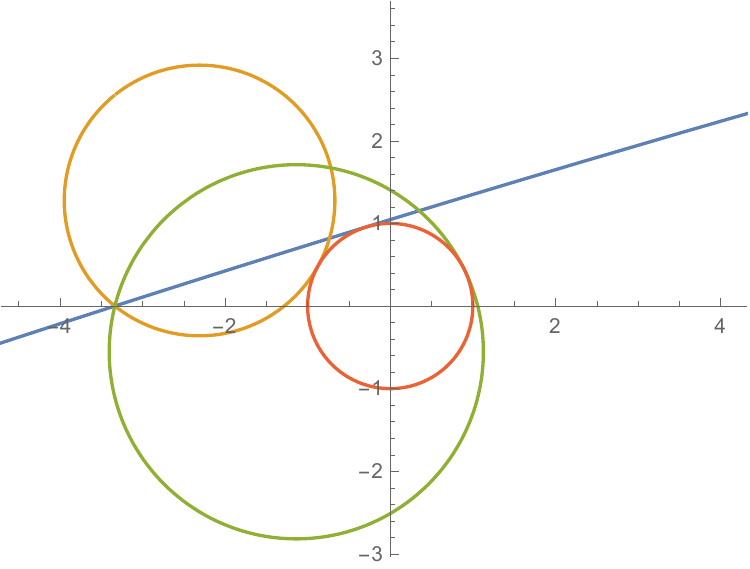}\ \caption{Straight line and circles which are preimages of some supporting lines of the unit disk under the Moebius mapping $w=\frac{z+a}{1+az}$,  $a=0.3$.}
\label{support}
\end{figure}

(b) It remains to consider the case of internal tangency. In this case,  $\Gamma_{\theta}$ is a circle passing through the points $-1/a$ and $e^{i\theta}$ and containing the unit circle in its interior.
Let $z_0$ be the center of $\Gamma_{\theta}$. Then the origin, as the center of the unit circle, is on the segment with endpoints $e^{i\theta}$ and $z_0$. Therefore, $z_0=-te^{i\theta}$ for some $t>0$ and the radius of $\Gamma_{\theta}$ equals $R=t+1$. Since  the point $-1/a\in \Gamma_{\theta}$, we have $|z_0+1/a|=R$ or
\begin{equation}\label{R1}
|(1-R)e^{i\theta}+1/a|=R.
\end{equation}
Therefore, by the triangle inequality, $R\ge1/a-(R-1)$, consequently,
\begin{equation}\label{R}
R\ge \frac{1}{2}\left(1+\frac{1}{a}\right).
\end{equation}

Taking into account Theorem~\ref{supp}, we will show that
\begin{equation}\label{a}
\th\frac{\rho_{{D}_\phi}(z_1,z_2)}{2}\le (1+a)s_{\mathbb{U}}(z_1,z_2)=(1+a)\th\frac{\rho_{\mathbb{H}_\psi}(z_1,z_2)}{2},
\end{equation}
where $e^{i\psi}$ is the point at which the maximal ellipse with foci at the points $z_1$ and $z_2$, contained in $\mathbb{U}$, touches the unit circle.

By \eqref{K},
$$
\th\frac{\rho_{{D}_\phi}(z_1,z_2)}{2}=\frac{R|z_1-z_2|}{|R^2-(z_1-z_0)(\overline{z}_2-\overline{z}_0)|}=\frac{R|z_1-z_2|}{|R^2-(e^{-i\theta}z_1+R-1)(e^{i\theta}\overline{z}_2+R-1)|}$$
$$=\frac{R|z_1-z_2|}{|2R-1-(R-1)(e^{-i\theta}z_1+e^{i\theta}\overline{z}_2)-z_1\overline{z}_2|}.
$$
Besides,
$$\th\frac{\rho_{\mathbb{H}_\psi}(z_1,z_2)}{2}=\frac{|z_1-z_2|}{|2-e^{-i\psi}z_1-e^{i\psi}\overline{z}_2|}$$
and from \eqref{R} we have
\begin{equation}\label{1a}
1+a\ge\frac{2R}{2R-1}.
\end{equation}

Taking into account \eqref{1a}, we see that  \eqref{a}  will be proved if we will show that
\begin{equation}\label{2R}
\frac{R|z_1-z_2|}{|2R-1-(R-1)(e^{-i\theta}z_1+e^{i\theta}\overline{z}_2)-z_1\overline{z}_2|}\le \frac{2R}{2R-1}\,\frac{|z_1-z_2|}{|2-e^{-i\psi}z_1-e^{i\psi}\overline{z}_2|}
\end{equation}
or
\begin{equation}\label{b}
\left|1-\frac{R-1}{2R-1}\,(e^{-i\theta}z_1+e^{i\theta}\overline{z}_2)-\frac{z_1\overline{z}_2}{2R-1}\right|\ge\frac{1}{2}\,|2-e^{-i\psi}z_1-e^{i\psi}\overline{z}_2|.
\end{equation}

Denote $\zeta_k(\varphi)=1-e^{-i\varphi}z_k$, $k=1$, $2$. According to Theorem~\ref{supp}~(II),  the points $\zeta_1(\psi)$ and $\overline{\zeta_2(\psi)}$ have the same argument, consequently,  $|\zeta_1(\psi)+\overline{\zeta_2(\psi)}|=|\zeta_1(\psi)|+|\overline{\zeta_2(\psi)}|$.
We rewrite \eqref{b} in the form
$$
\left|\frac{R-1}{2R-1}(2-e^{-i\theta}z_1-e^{i\theta}\overline{z}_2)+\frac{1-z_1\overline{z}_2}{2R-1}\right|\ge\frac{1}{2}\,|2-e^{-i\psi}z_1-e^{i\psi}\overline{z}_2|
$$
or, taking into account that $$1-z_1\overline{z}_2=1-e^{-i\psi}z_1e^{i\psi}\overline{z}_2=1-(1-\zeta_1(\psi))(1-\overline{\zeta_2(\psi)})=\zeta_1(\psi)+\overline{\zeta_2(\psi)}-\zeta_1(\psi)\overline{\zeta_2(\psi)},$$
as
\begin{equation}\label{3}
\left|\frac{R-1}{2R-1}(\zeta_1(\theta)+\overline{\zeta_2(\theta)})+\frac{\zeta_1(\psi)+\overline{\zeta_2(\psi)}}{2R-1}-\frac{\zeta_1(\psi)\overline{\zeta_2(\psi)}}{2R-1}\right|\ge\frac{1}{2}\,|
\zeta_1(\psi)+\overline{\zeta_2(\psi)}|.
\end{equation}
Denote $r=|
\zeta_1(\psi)+\overline{\zeta_2(\psi)}|$,
$$
A=\frac{R-1}{2R-1}(\zeta_1(\theta)+\overline{\zeta_2(\theta)})+\frac{\zeta_1(\psi)+\overline{\zeta_2(\psi)}}{2R-1}\,, \quad B=\frac{\zeta_1(\psi)\overline{\zeta_2(\psi)}}{2R-1}\,.
$$
 Since the curve  $z=\zeta_1(\phi)+\overline{\zeta_2(\phi)}$, $0\le\phi\le 2\pi$, is an ellipse, i.e. convex smooth  curve, and the modulus of $\zeta_1(\phi)+\overline{\zeta_2(\phi)}$ attains its maximal value at the point $\phi=\psi$, we conclude that the ellipse is outside of the disk of radius $r$, centered at the origin, and touches its boundary circle at some point $e^{i\eta}r$. Then the ellipse is in the half-plane $\{z\in \mathbb{C}\mid \Re (e^{-i\eta}z)\ge r\}$.   This implies
$$
\Re (e^{-i\eta}A)=\frac{R-1}{2R-1}\Re [e^{-i\eta}(\zeta_1(\theta)+\overline{\zeta_2(\theta)})]+\frac{1}{2R-1}\Re [e^{-i\eta}{\zeta_1(\psi)+\overline{\zeta_2(\psi)}}]
\ge \frac{Rr}{2R-1},
$$
thus,
\begin{equation}\label{c}
|A|\ge \frac{Rr}{2R-1}\,.
\end{equation}
At last, since the points $\zeta_1(\psi)$ and $\overline{\zeta_2(\psi)}$  have the same argument we obtain
$$
|\zeta_1(\psi)\overline{\zeta_2(\psi)}|\le \frac{1}{2}\, |\zeta_1(\psi)+\overline{\zeta_2(\psi)}|=\frac{1}{2}\,r,
$$
therefore,
\begin{equation}\label{d}
|B|\le \frac{r}{2(2R-1)}\,.
\end{equation}
From \eqref{c}, \eqref{d} and the triangle inequality we obtain
$$
|A-B|\ge|A|-|B|\ge  \frac{Rr}{2R-1}\,-\frac{r}{2(2R-1)}\,=\frac{r}{2}\,,
$$
therefore, we have inequality \eqref{3}, equivalent to~\eqref{a}.\medskip

\section{Corollaries of the main theorem}

\begin{corollary}\label{cor1}
Under the assumptions of Theorem~\ref{th1}  we have
\begin{equation}\label{main1}
(1+a)^{-1}s_\mathbb{U}(z_1,z_2)\le s_\mathbb{U}(f(z_1),f(z_2))\le (1+a)s_\mathbb{U}(z_1,z_2), \quad z_1, z_2\in \mathbb{U},
\end{equation}
and the constants $(1+a)^{-1}$ and $(1+a)$  are the best possible.
\end{corollary}

\begin{corollary}\label{cor2}
Let $F:\mathbb{B}^n\to \mathbb{B}^n$  be a Moebius automorphism of the unit ball $\mathbb{B}^n$ in $\mathbb{R}^n$ and $F(0)=a$. Then
$$(1+|a|)^{-1}s_{\mathbb{B}^n}(z_1,z_2)\le s_{\mathbb{B}^n}(F(z_1),F(z_2))\le (1+|a|)s_{\mathbb{B}^n}(z_1,z_2), \quad z_1, z_2\in \mathbb{B}^n,$$
where  $|a|$ is the Euclidean norm of $a$ in $\mathbb{R}^n$.
\end{corollary}

\begin{proof} Actually, by \cite[thrm~3.5.1~(i)]{beardon}, any Moebius automorphism of $\mathbb{B}^n$ is a composition of an orthogonal mapping and a reflection with respect to a sphere orthogonal to the boundary of $\mathbb{B}^n$. Since orthogonal mappings keep $s$-metric and for every $u$, $v\in \mathbb{B}^n$ any reflection keeps the two-dimensional plane $L$ containing $u$, $v$ and $0$, the problem reduces to considering the Moebius transformation in the plane $L$ and applying Theorem~\ref{main} and its Corollary~\ref{cor1}.

\end{proof}

The following corollary improves the statement of Theorem~\ref{th1}.

\begin{corollary}\label{cor3}
Let the assumptions of Theorem~\ref{th1}  hold and $e^{i\phi}$ be the contact point  of the unit circle and the maximal ellipse with foci at the points $w_1$ and $w_2$, contained in~$\mathbb{U}$. If $\cos\phi\ge a$, then
$$
s_\mathbb{U}(w_1,w_2)\le \left(1+a\,\frac{\cos\phi-a}{1-a\cos\phi}\right)\,s_\mathbb{U}(z_1,z_2),
$$
and if $\cos\phi< a$, then
$$
s_\mathbb{U}(w_1,w_2)\le s_\mathbb{U}(z_1,z_2).
$$
\end{corollary}

\begin{proof}
In the proof of Theorem~\ref{th1} we use estimate \eqref{R} and its corollary \eqref{1a} for the radius $R$. But if we know the contact point $e^{i\phi}$, we can  use the exact value of $R$. From \eqref{R1} we find
$$
R=\frac{1+2a\cos\theta+a^2}{2a(a+\cos \theta) }\,.
$$
where $e^{i\theta}=f^{-1}(e^{i\varphi})$.
Then, instead of $(1+a)$ in \eqref{main}, we obtain the constant
$$
\frac{2R}{2R-1}=1+a\,\frac{a+\cos\theta}{1+a\cos \theta}\,.
$$
Since
$$
e^{i\theta}=\frac{e^{i\phi}-a}{1-e^{i\phi }a}\,,
$$
we obtain
$$
\cos\theta= \frac{(1+a^2)\cos\phi-2a}{1+a^2-2a\cos \phi}\,,
$$
therefore,
$$
\frac{2R}{2R-1}=1+a\,\frac{\cos\phi-a}{1-a\cos\phi}\,.
$$
In the case $\cos\phi< a$ we have the external tangency of circles. In the proof above it was showed that then \eqref{2a} holds.
\end{proof}

\section{ACKNOWLEDGEMENT}

The author thanks Prof M.~Vuorinen for valuable remarks.

\section{FUNDING}
This work was financially supported by the Russian Science Foundation (grant No~23--11--00066).

\section{CONFLICT OF INTEREST}
The author of this work declare that he has no conflicts of interest.

\end{document}